\documentclass[a4paper,10pt]{article}
\usepackage[utf8]{inputenc}
\usepackage{amsmath,amsthm,amssymb,amsfonts}
\usepackage{xcolor}
\usepackage{accents}
\usepackage[hidelinks]{hyperref}
\usepackage{multicol}
\usepackage{mathtools}
\usepackage{bbold}
\newtheorem{theorem}{Theorem}[section]
\newtheorem{definition}[theorem]{Definition}

\newtheorem{corollary}[theorem]{Corollary}
\newtheorem{lemma}[theorem]{Lemma}
\newtheorem{example}[theorem]{Example}
\newtheorem{remark}[theorem]{Remark}

\def\FF{\mathbb F}

\def\diag{\mathop{\mathrm {diag}}}
\def\R{\mathcal R}
\def\C{\mathcal C}
\def\S{\mathcal S}
\def\dotS{\accentset{\bullet}{\S}}
\def\per{\mathop{\mathrm{ per}}}
\title{Additive preservers of permanent rank}
\author{Alexander Guterman, Bojan Kuzma\thanks{This work is supported in part by the Slovenian Research Agency (research program P1-0285 and research projects 
J1-50000,
N1-0296,
N1-0428,
J1-70047, and 
J1-70046)}, Leonid Ovchinnikov
}
\date{}
\begin{document}
\newcommand{\F}{\mathbb{F}}
\newcommand{\Fp}{\mathbb{F}_p}
\newcommand{\Z}{\mathbb{Z}}
\newcommand{\M}[1]{\operatorname{Mat}_{#1}(\F)}
\newcommand{\charac}{\operatorname{char}}
\newcommand{\prk}{\operatorname{prk}}
\newcommand{\T}{\mathrm{T}}
\renewcommand{\Im}{\mathop{\mathrm{Im}}}
\maketitle
\begin{abstract}
The permanent rank of a matrix $A$ is the size of the maximal square submatrix in $A$ which has a nonzero permanent. In this paper we characterize additive transformations $\Phi$  which preserve matrices of per-rank-one. Under an additional assumption that  $\Phi$ is surjective or that  it preserves per-rank-one in both directions 
we prove that $\Phi$ is a composition of a multiplication with diagonal matrices and permutation matrices from both sides, transposition, and injective endomorphism of the base field.


\end{abstract}
\section{Introduction and Preliminaries}
The theory of linear preservers originated with Frobenius \cite{Frobenius1897}, who characterized determinant-preserving maps, and was extended by Dieudonné \cite{Dieudonne1949} through his work on singular matrix preservers. These foundational results sparked extensive research on various matrix invariants and their preservers, as documented in \cite{Pierce1992}.

        This paper is concerned with permanent rank, a matrix invariant defined via the permanent function, which serves as an analogue of the ordinary rank. Recall that the permanent is defined similarly  to the determinant: we sum up all generalized diagonals of a matrix, but, unlike the determinant, we do not alternate  signs. More precisely, for a square matrix $A=(a_{ij})_{ij}$ of size $n$-by-$n$, its permanent equals $\per( A)=\sum_{\sigma\in S_n} \prod_{i=1}^n a_{i\sigma(i)}$, where we sum   over all permutations of the set $[n]:=\{1,2,\dots,n\}$. We remark in passing that the permanent naturally arises in combinatorics~\cite{MincMarcus}:  Given a bipartite graph $G$,  the number of perfect matchings equals the permanent of the corresponding adjacency matrix. The permanent rank of a matrix was introduced by Yu~\cite{Yu99} and was used to give a probabilistic  (``almost surely'') solution to the Alon-Jaeger-Tarsi conjecture,  \cite{Alon1989,Jaeger}. Its abstract  properties are studied for example in~\cite{KSh}.
        
    We now recall the definition of the permanent rank from~\cite{Yu99}.
    \begin{definition}
    Let $A \in \M{n}$ be a matrix.  The permanent rank, $\prk(A)$, of the matrix $A$ is the size of the maximal square submatrix in $A$ with nonzero permanent.
    \end{definition}
    In contrast to the usual rank, permanent rank is not invariant under general invertible transformations, making its preservers more rigid and structurally restricted.
    
    While preservers of ordinary rank have been studied much more extensively  (see \cite{Akhmedova2025,Beasley,Botta,Kuzma2002,Loewy,MarcusMoyls,MarcusMoyls2}) the case of permanent rank preservers is relatively new. For instance, the paper \cite{Guterman2024} investigates linear preservers of permanent rank, providing  initial results in this area.
    
        Throughout these notes, we assume that $n \geq 3$ and $\charac(\F) \neq 2$.
        This restriction is clearly necessary since for  fields of characteristic two, the permanent rank coincides with the usual rank, and additive  preservers of rank are already classified,  see \cite{Kuzma2002}. Let us denote the sets of matrices with fixed permanent rank, $\Lambda^k$, and with permanent rank bounded from above, $\Lambda^{\leq k}$, as follows.
         \begin{definition} Let $k\in\{1,2,\dots,n\}$ be an integer.
        \begin{align*}
            \Lambda^k &\coloneqq \{ A \in \M{n} \mid \prk(A) = k \} \\
            \Lambda^{\leq k} &\coloneqq \{ A \in \M{n} \mid \prk(A) \leq k \}
        \end{align*}
    \end{definition}
    For a transformation $\sigma\colon[n]\to[n]$ we let
    $$P_\sigma:=\sum_{i=1}^n E_{i\sigma(i)} \mbox{, \ \  i.e., \ \ }  p_{ij} =
    \begin{cases}
        1, & \text{if } j = \sigma(i) \\
        0, & \text{otherwise.}
    \end{cases}
    $$
    and call it \emph{a generalized permutation matrix induced by $\sigma$}.
    For example, if $\sigma(i)=1$ for each $i$, then
    $$P_\sigma=\begin{pmatrix}
                1 &0&\dots &0\\
                \vdots & \vdots&\ddots & \vdots\\
                1&0&\dots&0
               \end{pmatrix}.$$
Notice that, if $\sigma$ is injective, hence a permutation, then
    $P_\sigma$ is  the permutation matrix, i.e., a matrix with exactly one entry equal to  $1$ in each row and column and all other entries equal to~$0$.    Let us introduce some further notation for the special subsets of $\M{n}$ that we use throughout.  Given  integers $i,j\in[n]:=\{1,\dots,n\}$ and a nonzero $\alpha\in\F$  we call a matrix of the form $\alpha E_{ij}$ \emph{a weighted matrix unit}; the  additive subset (actually, it is even a linear subspace) of the form
    $$\F E_{ij}$$
    \emph{a cell}, and we call additive subsets (actually, linear subspaces) of the form
\begin{align*}
 \R_i&:=\{A=(a_{st})\in \M{n};\;\;a_{st}=0\hbox{ for all } s\neq i\}\\
 \C_j&:=\{A=(a_{st})\in \M{n};\;\;a_{st}=0\hbox{ for all } t\neq j\}
\end{align*}
\emph{an $i$-th row} and \emph{a $j$-th column}, respectively. By  a  \emph{line} we mean either (an unspecified) row or (an unspecified) column. A \emph{full square} is an additive subset (actually, a subspace) of the form
$$\S^{ij}_{st}:=\{\alpha (E_{is}+E_{it})+\beta(E_{js}-E_{jt});\;\;\alpha,\beta\in\FF\}$$
for some $i,j,s,t$ with $i<j$ and $s<t$ (it occupies rows $i,j$ and columns $s,t$), or its transpose, i.e., a space of  the form
$$\dotS^{ij}_{st}:=\{\alpha (E_{is}+E_{js})+\beta(E_{it}-E_{jt});\;\;\alpha,\beta\in\FF\}=\{A^T;\;\ A\in\S^{st}_{ij}\}.$$
Notice that there are $2n$ different lines ($n$ rows and $n$ columns) and there are $2{n\choose 2}^2$ different full squares. \emph{A square} is a subset of a full square; the square is called \emph{nondegenerate} if it is not contained in a row or a column. By  Guterman and Spiridonov  \cite[Lemma 2.3]{Guterman2020} (cf. Lemma~\ref{lemma:prk1_structure} below) these are exactly all the maximal additive sets  consisting of  matrices of per-rank at most one, modulo applying the standard maps.
A sum of $k\ge2$  full squares with pairwise disjoint rows and columns will be called their  direct sum and denoted with $\oplus$. This is well defined only if the size of matrices is at least $2k$.
For example,
$$\S^{12}_{12}=\{\left(\begin{smallmatrix}\alpha &\alpha\\
                      \beta&-\beta                     \end{smallmatrix}\right)\oplus 0_{n-2};\;\;\alpha,\beta\in\FF\}$$
$$\S^{12}_{12}\oplus
\dotS^{34}_{34}=\left\{\left(\begin{smallmatrix}\alpha &\alpha\\
                      \beta&-\beta                     \end{smallmatrix}\right)\oplus\left(\begin{smallmatrix}\gamma &\delta\\
\gamma&-\delta                     \end{smallmatrix}\right)\oplus 0_{n-4};\;\;\alpha,\beta,\gamma,\delta\in\FF\right\},$$
however, neither $\S^{12}_{12}+\dotS^{13}_{45}$ nor $\S^{12}_{12}+\S^{13}_{45}$ is  a direct sum of two full squares, because in both cases the two full squares share the first row.

   Our goal is to  characterize additive maps which preserve matrices of per-rank-one (see Theorem~\ref{th:prk1_main}). 
   Given a matrix $A=(a_{ij})_{ij} 
    \in \M{n}$  and a field endomorphism (i.e., an additive and multiplicative function $\varphi\colon\FF\to\FF$),  we let $A^\varphi$  be a matrix obtained from $A$ by applying $\varphi$ entry-wise, that is,
     $$A^\varphi = \bigl(\varphi(a_{ij})\bigr)_{ij}.$$
    \begin{theorem}
        \label{th:prk1_main} Let $\FF$ be a field of characteristic different from two, let $n\ge 3$ and let $\Phi\colon \M{n} \to \M{n}$ be an additive map such that
        $$\Phi(\Lambda^{1}) \subseteq \Lambda^{1}.$$ Then either
        $$\Phi(\M{n})\subseteq\R_i\  \hbox{ or }\  \Phi(\M{n})\subseteq\C_j\ \hbox{ or } \ \Phi(\M{n})\subseteq (\R_i+\R_{i'})\cap (\C_j+\C_{j'})$$
        for some $i,j,i',j'\in[n]$         or else  there exist transformations $\sigma, \tau \colon[n]\to[n]$, diagonal matrices $D_1, D_2 \in \M{n}$ with nonzero diagonal entries and a nonzero field endomorphism $\varphi$ of $\F$ such that
                \begin{equation}
        \label{eq:form1}
        \Phi(A) =  P_{\sigma}^T D_1A^\varphi D_2 P_{\tau}  \quad \forall A \in \M{n}
        \end{equation}
        or
        \begin{equation}
        \label{eq:form2}
        \Phi(A) =P_{\sigma}^T D_1  (A^\varphi)^T D_2P_{\tau}  \quad \forall A \in \M{n}.
        \end{equation}
      \end{theorem}
Let us show that the forms \eqref{eq:form1}--\eqref{eq:form2} might define  a per-rank-one preserver even if $P_\sigma$ or $P_\tau$ are not permutation matrices.
 \begin{example}
Let $\FF$ be a field such that there exists a nonsurjective nonzero endomorphism $\varphi\colon\FF\to\FF$ (for example, the field of complex numbers has this property, as well as the field of rational functions in a single indeterminate over some ambient field -- the nonsurjective endomorphism is a squaring of the indeterminate). Choose $w\in\FF\setminus\varphi(\FF)$,  define
$$D:=\diag(w,1,\dots,1)$$
and let
$$\sigma\colon i\mapsto\begin{cases}
                        1; & i=1 \hbox{ or } i=2\\
                        i; & i\ge 3
                       \end{cases}$$
be  a  transformation on $[n]$.  Then, the map
$$\Phi\colon A\mapsto P_\sigma^T DA^\varphi$$
preserves per-rank-one even though $P_\sigma$ is not a permutation matrix. 
Namely, for $A=(a_{ij})_{ij}$ we have 
$$\Phi(A)=\begin{pmatrix} wa_{11}^\varphi+a_{21}^\varphi & wa_{12}^\varphi+a_{22}^\varphi &\dots\\
0 & 0 & \dots\\
a_{31}^\varphi & a_{32}^\varphi & \dots\\
\vdots &&\ddots
\end{pmatrix}$$
and if all $2$-by-$2$ per-minors of $A$ vanish, then clearly the same holds for per-minors of $\Phi(A)$ which do not contain the first row. But even $2$-by-$2$ per-minors of the first and, say, the third row  vanish since they equal 
$$\begin{aligned}
    \per\begin{pmatrix}
    wa_{1i}^\varphi+a_{2i}^\varphi &wa_{1j}^\varphi+a_{2j}^\varphi\\
    a_{3i}^\varphi & a_{3j}^\varphi
\end{pmatrix}&=w(a_{1i}^\varphi a_{3j}^\varphi+a_{1j}^\varphi a_{3i}^\varphi)+(a_{2i}^\varphi a_{3j}^\varphi+a_{2j}^\varphi a_{3i}^\varphi)\\
&=w\per \begin{pmatrix} a_{1i}^\varphi & a_{1j}^\varphi\\
a_{3i}^\varphi & a_{3j}^\varphi
\end{pmatrix}+\per\begin{pmatrix} a_{2i}^\varphi & a_{2j}^\varphi\\
a_{3i}^\varphi & a_{3j}^\varphi
\end{pmatrix}=0.\end{aligned}$$
Also, since $w\in\FF\setminus\varphi(\FF)$, the first row of $\Phi(A)$ vanishes if and only if the first two rows of $A$ vanish, while each other row of $\Phi(A)$ vanishes if and only if it vanishes in $A$. Hence, $\Phi(A)=0$ if and only if $A=0$. Thus, if $A$ is of per-rank one, then the same holds for $\Phi(A)$.
\end{example}
\begin{remark}\label{rem:phisurjective}
Such example is impossible if the field~$\FF$ does not allow nonzero nonsurjective endomorphisms (the field of rational or real numbers,  any finite field, or  the field of $p$-adic numbers have this property; see, e.g., \cite[Corollary 4.2]{Conrad}, or \cite{Wagner1974}, whose proof for automorphisms extends verbatim to endomorphisms).  
Namely, if $\sigma$ is noninjective, say, $\sigma(i_1)=\sigma(i_2)$ for $i_1\neq i_2$, consider a matrix $A=E_{i_11}- w E_{i_21}\in\Lambda^1$, where $w$ is such that $$\varphi(w) =\tfrac{d_{i_1}}{d_{i_2}}\in\FF=\varphi(\FF); \qquad D_1=\diag(d_{1},\dots,d_{n}).$$
 Then,
 \begin{align}
   \Phi(A)&= P_\sigma^T D_1 (E_{i_11}- \varphi(w) E_{i_21})D_2P_\tau\\
   &=\sum E_{\sigma(k)k}(d_{i_1}E_{i_11}- \varphi(w) d_{i_2}E_{i_21})D_2P_\tau\\
   &=(d_{i_1}E_{\sigma(i_1)1}-\varphi(w)d_{i_2}E_{\sigma(i_2)1})D_2P_\tau=0\notin\Lambda^1
 \end{align}
 Likewise we show that $\tau$ is injective, so both $P_\sigma$ and $P_\tau$ are permutation matrices.
\end{remark}
If we add more regularity on $\Phi$, then we can expect better results.
\begin{theorem}\label{th:prk1_maina} Under the assumptions of Theorem~\ref{th:prk1_main}, assume further that $\Phi$ is surjective or that  it preserves per-rank-one in both directions. 
Then, $\Phi$ takes the forms~\eqref{eq:form1}--\eqref{eq:form2} with injective  transformations $\sigma$ and $\tau$  
so $P_\sigma$ and $P_\tau$ are permutation matrices.
\end{theorem}
For surjective maps  Theorem~\ref{th:prk1_maina} follows from Remark~\ref{rem:phisurjective}. However, we prefer to give a self-contained short proof, which does not require the machinery needed to prove Theorem~\ref{th:prk1_main}. This will be done immediately after we prove Theorem~\ref{th:prk1_main}. 

We end this section with one final example of nonstandard per-rank-one preservers whose range is contained in  a $2$-by-$2$ block. 
\begin{example}
If $n=2$ there exists a linear bijection which preserves per-rank-one in both directions and is not of standard form~\eqref{eq:form1}--\eqref{eq:form2}. Consider, e.g.,   
\[
    \Phi\colon\left(\begin{matrix} a & b \\ c & d \end{matrix}\right)\mapsto
\begin{pmatrix}
 a + b + c - d & a + b - c + d \\
 a - b + c + d & -a + b + c + d
\end{pmatrix}
\]
and notice that the permanent of the matrix on the right equals $4 (b c + a d)$. Hence, $\Phi$ preserves per-rank-one in both directions. One also sees it is invertible with inverse $\Phi^{-1}\colon \left(
\begin{smallmatrix}
 x & y \\
 u & v \\
\end{smallmatrix}
\right)\mapsto
\frac{1}{4} \left(\begin{smallmatrix}
 (u-v+x+y) &  (-u+v+x+y) \\
(u+v+x-y) &  (u+v-x+y) \\
\end{smallmatrix}\right)
$.

This can be generalized to higher dimensions as well. One can show that the following linear map  $\Phi\colon \M{n}\to \M{2}\oplus 0_{n-2}$ (defined on matrix units) preserves the set $\Lambda^{\le 1}$ of matrices of per-rank-one or zero.
\[
\bigl(\Phi(E_{ij})\bigr)_{ij}
=
\begin{pmatrix}

\left(\begin{smallmatrix}
1&1\\
1&-1
\end{smallmatrix}\right)
&
\left(\begin{smallmatrix}
1&1\\
-1&1
\end{smallmatrix}\right)
&
r_3\left(\begin{smallmatrix}
1&1\\
1&-1
\end{smallmatrix}\right)
&
r_4\left(\begin{smallmatrix}
1&1\\
1&-1
\end{smallmatrix}\right)
&
\cdots
\\[2ex]

\left(\begin{smallmatrix}
1&-1\\
1&1
\end{smallmatrix}\right)
&
\left(\begin{smallmatrix}
-1&1\\
1&1
\end{smallmatrix}\right)
&
r_3\left(\begin{smallmatrix}
1&-1\\
1&1
\end{smallmatrix}\right)
&
r_4\left(\begin{smallmatrix}
1&-1\\
1&1
\end{smallmatrix}\right)
&
\cdots
\\[2ex]

c_3\left(\begin{smallmatrix}
1&1\\
1&-1
\end{smallmatrix}\right)
&
c_3\left(\begin{smallmatrix}
1&1\\
-1&1
\end{smallmatrix}\right)
&
c_3r_3\left(\begin{smallmatrix}
1&1\\
1&-1
\end{smallmatrix}\right)
&
c_3r_4\left(\begin{smallmatrix}
1&1\\
1&-1
\end{smallmatrix}\right)
&
\cdots
\\[2ex]

c_4\left(\begin{smallmatrix}
1&1\\
1&-1
\end{smallmatrix}\right)
&
c_4\left(\begin{smallmatrix}
1&1\\
-1&1
\end{smallmatrix}\right)
&
c_4r_3\left(\begin{smallmatrix}
1&1\\
1&-1
\end{smallmatrix}\right)
&
c_4r_4\left(\begin{smallmatrix}
1&1\\
1&-1
\end{smallmatrix}\right)
&
\cdots
\\[2ex]

\vdots&\vdots&\vdots&\vdots&\ddots
\end{pmatrix}.
\]
By comparing  the dimensions of its domain and its range, such linear map cannot be injective if $n\ge 3$. However, suppose $\Phi$ is semilinear under some nonsurjective field endomorphism $\sigma\colon\FF\to\FF$ such that $\FF/\sigma(\FF)$ is a transcendental field extension with  $t\in\FF$ as indeterminate over $\sigma(\FF)$. Here, we can   choose $c_i=t^{i}$ and $r_j=t^{(n+1)^j}$. Then, due to the uniqueness of integer expansion in basis $n+1$, $c_ir_j=c_{i'}r_{j'}$ if and only if $(i,j)=(i',j')$ so $\Phi$ is even injective.  However, it may not preserve $\Lambda^1$ in both directions.
\end{example}
    \section{Auxiliary results}
 The following  lemma is well-known. We omit its  straightforward proof.
    \begin{lemma}
    The permanent rank of a matrix is invariant under the following matrix operations.
    \begin{enumerate}
        \item Transposition, i.e. $A \mapsto A^\T$.
        \item Row permutations, i.e. $A \mapsto P_\sigma A$, where $\sigma \in S_n$.
        \item Column permutations, i.e. $A \mapsto A P_\tau$, where $\tau \in S_n$.
        \item Row nonzero rescaling, i.e. $A \mapsto D A$, where $D = \diag(d_1, \ldots, d_n) \in \M{n}$ is a diagonal matrix with nonzero diagonal entries $d_i \neq 0$.
        \item Column nonzero rescaling, i.e. $A \mapsto A D$, where $D = \diag(d_1, \ldots, d_n) \allowbreak \in \M{n}$ is a diagonal matrix with nonzero diagonal entries $d_i \neq 0$.
        \item Injective field endomorphisms, i.e. $A \mapsto A^\varphi$ where $\varphi$ is an injective endomorphism of $\F$ and $(A^\varphi)_{ij} = \varphi(A_{ij})$.
    \end{enumerate}
    \end{lemma}
   \noindent  The mappings of the first five types will hereafter be referred to as \emph{standard}.

We now present several results on maximal additive subsets of matrices of bounded permanent rank. Here, "additive" means closed under matrix addition, and "maximal" indicates that the subset is not properly contained in any other additive subset whose permanent rank does not exceed the same upper bound.  Understanding the structure of these subsets, as well as how our mapping acts on them, will be an essential tool in the subsequent analysis. We should mention that among the next four results we only require Lemma~\ref{lem:existence}, however we decided to present them in full 
for our future use.
    
To prove the next lemma we will use the following result by Alon \cite{Alon1999}.
    \begin{lemma} \cite[Lemma 2.1]{Alon1999}
    \label{lemma:Alon}
    Let $f = f(x_1, x_2, \ldots, x_n)$ be a polynomial in $n$ variables over an arbitrary field $\F$. Suppose that the degree of $f$ as a polynomial in $x_i$ is at most $t_i$ for $1 \leq i \leq n$, and let $S_i \subset \F$ be a set of at least $t_i + 1$ distinct members of  $\F$. If $f(x_1, x_2, \ldots, x_n) = 0$ for all $n$-tuples $(x_1, x_2, \ldots, x_n) \in S_1 \times S_2 \times \cdots \times S_n$, then $f \equiv 0$.
    \end{lemma}
    \begin{lemma}\label{lem:VS}
    Let $\mathcal{A}$ be an additive  subset of $\Lambda^{\leq k}$ and $\charac(\F) = 0$ or $\charac(\F) > k+2$. Then $\alpha_1 A_1 + \alpha_2 A_2 + \cdots + \alpha_m A_m$ belongs to $\Lambda^{\leq k}$ for all $\alpha_1, \ldots, \alpha_m \in \F$ and $A_1, \ldots, A_m \in \mathcal{A}$.
    \end{lemma}

    \begin{proof}
    Fix a natural number $m$ and matrices $A_1, A_2, \ldots, A_m \in \mathcal{A}$. Let $f \colon \mathbb{F}^m \to \mathbb{F}$ be defined by
\[
f(\alpha_1, \alpha_2, \ldots, \alpha_m) \coloneqq \operatorname{per}\bigl((\alpha_1 A_1 + \alpha_2 A_2 + \cdots + \alpha_m A_m)[I|J]\bigr),
\]
where $(\cdot)[I|J]$ denotes the $(k+1)\times(k+1)$ submatrix on rows $I$ and columns $J$, for some fixed index sets $I, J \subseteq [n]$ with $|I|=|J|=k+1$. Our goal is to show that $f \equiv 0$.

    Note that $f$ is a polynomial of degree at most $k + 1$ which, by additivity of $\mathcal{A}$, vanishes when variables are specialized to nonzero elements from the prime subfield ${\mathbb k}\subseteq\F$, here ${\mathbb k}= \F_p$ (or $\mathbb Z/p\mathbb Z$) if the characteristic $\charac(\F) =p$, and ${\mathbb k}= \mathbb Q, $ if $\charac(\F) =0$. 
     Since $\charac(\F) = 0$ or $\charac(\F) > k+2$, the prime  subfield  $\mathbb k$ contains at least $k + 2$ nonzero elements.
     It then follows from Lemma~\ref{lemma:Alon} that $f \equiv 0$.
    \end{proof}
    \begin{corollary} \label{cor:max}
    Let $k\in[n]$ and let the characteristic of $\F$ be as in Lemma~\ref{lem:VS}. Maximal  additive subsets in $\Lambda^{\le k}$ are $\FF$-linear subspaces.
    \end{corollary}
    \begin{lemma}
    \label{lem:existence}
    Every additive subset  of $\Lambda^{\leq k}$ and in particular every   matrix $A \in \Lambda^{\leq k}$ is contained in a maximal additive subset of $\Lambda^{\leq k}$.
    \end{lemma}
    \begin{proof}
    Apply Zorn's lemma.
    \end{proof}
    \section{Permanent rank one preservers}
  In this section we prove Theorem~\ref{th:prk1_main} and Theorem~\ref{th:prk1_maina}; without further notice, $\FF$ will be a field with characteristic different from two.
We start by giving a complete description (up to standard maps) of matrices of permanent 
rank one. To this end, we rely on the following result obtained by Guterman and 
Spiridonov~\cite{Guterman2020}.
    \begin{lemma} \cite[Lemma 2.3]{Guterman2020}
    \label{lemma:prk1_structure}
    Let $A \in \M{n}$. Then $A \in \Lambda^1$ if and only if at least one of the following three conditions holds.
    \begin{enumerate}
        \item $A$ has exactly one nonzero row.
        \item $A$ has exactly one nonzero column.
        \item After a  permutation of rows and columns, $A$ has the form:
        \[
        \begin{pmatrix} a_{11} & a_{12} & 0 & \cdots & 0 \\ a_{21} & a_{22} & 0 & \cdots & 0 \\ 0 & 0 & 0 & \cdots & 0 \\ \vdots & \vdots & \vdots & \ddots & \vdots \\ 0 & 0 & 0 & \cdots & 0 \end{pmatrix}, \; \text{where} \;\; a_{11} a_{22} + a_{12} a_{21} = 0,
        \]
        and not all elements $a_{11},a_{12},a_{21},a_{22}$   are zero.
    \end{enumerate}
    \end{lemma}
    From this lemma, one can derive the following classification:
    \begin{lemma}\label{lem-structure-of-prk1}
    Let $A \in \M{n}$ be a matrix of permanent rank one. Then, up to standard maps, $A$ 
takes one of the following forms:
    $$T_0 = \begin{pmatrix} 1 & 1\\
    1 & -1
           \end{pmatrix}\oplus 0_{n-2};\quad T_1=E_{11},\quad 
           \dots,\quad T_n=E_{11}+E_{12}+\dots+E_{1n}.$$
    \end{lemma}
      \begin{remark}Note that  a matrix $T_i$
 cannot be brought to the form $T_j$
 via standard maps, whenever $i \neq j$.
 This follows from the fact that the pair $\{ r(A), c(A) \}$, where $r(A)$ and $c(A)$ denote the number of nonzero rows and columns of $A$, is invariant under standard maps.
    \end{remark}
     We will henceforth say $A$ \emph{is of type $T_i$} if $\mathrm{St}(A)=T_i$ for some (composition of) standard maps~$\mathrm{St}$.
 
    Next, we characterize maximal additive subsets of $\Lambda^{\leq 1}$. The following lemma from \cite{Guterman2020} will be useful.
    \begin{lemma} \cite[Lemma 3.4]{Guterman2020}\label{lem:Guterman2020}
    Consider an additive subset $\mathcal{A} \subset \Lambda^{\leq 1}$. Then one of the following holds:
    \begin{enumerate}
        \item $\mathcal{A} \subseteq \R_i$ for some $i$.
        \item $\mathcal{A} \subseteq \C_j$ for some $j$.
        \item $\mathcal{A} \subset (\R_{i}+\R_j)\cap (\C_s+\C_t)$ for some $i,j,s,t$.
    \end{enumerate}
    \end{lemma}
    Using these results, it can be verified with a straightforward, though somewhat tedious, calculation that:
    \begin{lemma}
    \label{lemma:max_subsets}
    Let $\mathcal{A}$ be a maximal additive subset of $\Lambda^{\leq 1}$.
    \begin{itemize}
        \item[(i)] If $T_0$ belongs to $\mathcal{A}$, then
        \begin{align*}
\mathcal{A} &= \left\{ \left(\begin{smallmatrix} \alpha & \alpha \\ \beta & -\beta 
\end{smallmatrix}\right)\oplus 0_{n-2}\mid \alpha, \beta \in \F \right\} = \S^{12}_{12} \\
&\text{or} \\
\mathcal{A} &= \left\{ \left(\begin{smallmatrix} \alpha & \beta \\ \alpha & -\beta 
\end{smallmatrix}\right)\oplus 0_{n-2}\mid \alpha, \beta \in \F \right\} = \dotS^{12}_{12}.
\end{align*}
        \item[(ii)] If $T_1$ belongs to $\mathcal{A}$, then $\mathcal{A} = \R_1$ or $\mathcal{A} = \C_1$.
        \item[(iii)] If $T_2$ belongs to $\mathcal{A}$, then $\mathcal{A} = \R_1$ or $\mathcal{A} =\S^{1i}_{12}$
        for some $i \neq 1$.
        \item[(iv)] If $T_j$, with $j > 2$, belongs to $\mathcal{A}$, then $\mathcal{A} = \R_1$.
    \end{itemize}
   \end{lemma}
      Suppose that $\Phi \colon \M{n} \to \M{n}$ is an additive map satisfying
    \[
    \Phi(\Lambda^1) \subseteq \Lambda^1.
    \]
    By Lemma~\ref{lemma:max_subsets} 
    the image of any line is either contained in a line or contained in a full square. We therefore distinguish two cases:
    \begin{enumerate}
        \item every line is mapped into a line;
        \item at least one line is mapped into a  square which is not contained in a line.
    \end{enumerate}
    We begin with the second case. Using standard maps in the domain and the codomain, we may assume without     loss of generality that
    \[
    \Phi(\R_1) \subseteq \S^{12}_{12}.
    \]
    In the following Lemmas \ref{lem:46}--\ref{lem:TODO}, it is implicitly assumed that $\Phi(\R_1)$ is not contained in any line.
    \begin{lemma}\label{lem:46}
    Let $\Phi \colon \M{n} \to \M{n}$ be an additive map such that $\Phi(\Lambda^1) \subseteq \Lambda^1$. Suppose that $\Phi(\R_1) \subseteq \S^{12}_{12}$. Then
    \[
  \text{for all } A \in \M{n} \quad  \text{ if }  i\ge3  \text{ and } j \ge3  \text{ then } (\Phi(A))_{ij} = 0.
    \]
    \end{lemma}
    \begin{proof}
    Assume, towards a contradiction, that there exist $A \in \M{n}$ and indices $i_0,j_0 \ge 3$ such that $(\Phi(A))_{i_0 j_0} \neq 0$. Since $\Phi$ is additive, there exists a weighted matrix unit $x E_{st}$ satisfying
    \[
    (\Phi(x E_{st}))_{i_0 j_0} \neq 0.
    \]
    Necessarily $s > 1$, because $\Phi(\R_1) \subseteq \S^{12}_{12}$. \\
    Applying standard maps, we may assume that
    \[
    (\Phi(E_{21}))_{33} \neq 0.
    \]
    Consider now the images of $E_{11}$ and $E_{21}$. These matrices belong to $\C_1$, so their images are contained in some maximal additive subset of~$\Lambda^{\le1}$. However, $\Phi(E_{11})$ and $\Phi(E_{21})$ together have at least three nonzero columns (the first, second, and third) and at least two nonzero rows (either the first and third or the second and third), which is impossible for maximal additive subsets of~$\Lambda^{\le1}$. This contradiction proves the claim.
    \end{proof}
    \begin{lemma}\label{lem:line-square}
   Under the assumptions of Lemma~\ref{lem:46} we actually have
   $\Phi(\M{n}) \subseteq \C_1  +  \C_2$.
    \end{lemma}
    \begin{proof}
    By the  Lemma~\ref{lem:46}, we already know that
    \[
    (\Phi(A))_{ij} = 0 \quad \text{for all } A \in \M{n} \text { and } i,j \ge3.
    \]
    It therefore suffices to show that the image of a weighted matrix unit cannot have a nonzero entry in position $(i,j)$ with $i \in \{1,2\}$ and $j \ge3$. 

    Assume the contrary. Using standard maps in both the domain and the codomain, we may assume that
    \[
    (\Phi(E_{21}))_{23} \neq 0;
    \]
    if a swap of the first two rows in the codomain is required to achieve this, we additionally multiply the second column by $-1$ in order to preserve the structure of the full square $\S^{12}_{12}\supseteq\Phi(\R_1)$. 

Then for a fixed $x\in\FF$
    \[
    \Phi(xE_{11}) = \begin{pmatrix} \alpha_0 & \alpha_0 \\ \beta_0 & -\beta_0 \end{pmatrix} \oplus 0.
    \]
    If $\alpha_0 \neq 0$, then $\Phi(\C_1)$, which contains  $\Phi(x E_{11})$ and $\Phi(E_{21})$,  would have nonzero entries in three columns and two rows, and hence would not be contained in any maximal additive subset of~$\Lambda^{\le1}$. Thus, $\alpha_0 = 0$ and
    consequently
    $$\Phi(\FF E_{11})\subseteq\Phi(\C_1) \subseteq \R_2.$$
    Since, by the assumptions, $\Phi(\R_1) \nsubseteq \R_2$, there exists a weighted matrix unit $x E_{1j}$ with $j \ge2$ such that
    \[
    \Phi(x E_{1j}) = \begin{pmatrix} \alpha_1 & \alpha_1 \\ \beta_1 & -\beta_1 \end{pmatrix} \oplus 0 \quad \text{with } \alpha_1 \neq 0.
    \]
    
    Now consider $\Phi(x E_{2j})$. It cannot have a nonzero entry in position $(2,3)$, since otherwise the image of $\C_j$ would have nonzero entries in three columns and two rows, and hence could not be contained in any maximal additive subset of~$\Lambda^{\le1}$.
    Moreover, it cannot have nonzero entries simultaneously in positions $(1,1)$ and $(1,2)$, since otherwise the image of $\R_2$, which contains $\Phi(xE_{2j})$ and $\Phi(E_{21})$, would have nonzero entries in three columns and two rows, and hence could not be contained in any maximal additive subset of $\Lambda^{\le1}$. \\
    Define per-rank-one matrices
    \[
    A_1 \coloneqq (E_{11} + x E_{1j}) + (E_{21} - x E_{2j}), \quad
    A_2 \coloneqq (2 E_{11} + 2x E_{1j}) + (E_{21} - x E_{2j}).
    \]
    Both $\Phi(A_1)$ and $\Phi(A_2)$ have a nonzero entry in position $(2,3)$, and at least one of them has nonzero entries in positions $(1,1)$ and $(1,2)$.  Hence as $\mathrm{char}\,\FF\neq2$, at least one of these two matrices has three nonzero columns and two nonzero rows, and therefore permanent rank is greater than~$1$, contradicting the assumption that $\Phi(\Lambda^1) \subseteq \Lambda^1$.
    \end{proof}
    \begin{lemma}\label{lem:line->squarepart2}
Under the assumptions of Lemma~\ref{lem:46} we actually have
$$\Phi(\M{n}) \subseteq (\R_1 + \R_2) \cap (\C_1 + \C_2).$$
\end{lemma}
\begin{proof}
 Assume otherwise. Notice first that we can assume $n\ge 3$, otherwise there is nothing to do. By composing $\Phi$ with column permutations and permutations of rows $3,\dots,n$ we can achieve that the image of a  weighted matrix unit $\Phi(x E_{21})$  from   $\Phi(\R_2)$   has a nonzero entry $\gamma$ at position $(3,1)$.
 Clearly, $\Phi(\R_2)$ cannot be contained in $\C_1$, for then,
 $$\Phi(x E_{21}+E_{11})=\left(\begin{smallmatrix}
     \ast+ \alpha & \alpha &0&\dots\\
      \ast+\beta & -\beta &0&\dots\\
      \gamma &0 &0&\dots\\
      \vdots &\vdots  && \ddots                                                                                     \end{smallmatrix}
      \right)$$
    (for some unspecified numbers $\ast$) would be of per-rank two, a contradiction.
    
    Assume  $\Phi(\R_2)$  is not contained in a line. Then, by Lemmas~\ref{lem-structure-of-prk1} and~\ref{lemma:max_subsets}, $\Phi(\R_2)$ is contained in some full square $\S$ which, due to the inclusion $\Phi(x E_{21})\in\Phi(\R_2)\subseteq\S$, must  occupy a position $(3,1)$. Also, by Lemma~\ref{lem:line-square}, all the positions which $\S$ occupies must lie in the first two columns.
   
    We will now use the elements from row $\R_2$ to force conditions on corresponding elements from row $\R_1$ to see that this case is  contradictory.  Choose  any nonzero weighted matrix unit $b E_{2j}\in \R_2$.
    By our assumption, $\Phi(\R_2)$ cannot lie in a line, so there exist a  (perhaps zero) weighted matrix unit $c E_{2j'}\neq b E_{2j}$, such that $\Phi(b E_{2j}+c E_{2j'})\in\S$  is  nonzero in two rows (one of them being the third row) and two columns.   Consider then a per-rank-one matrix
    $$A:= (b E_{1j}-c E_{1j'})+ \lambda(b E_{2j}+c E_{2j'});\quad \lambda\in {\mathbb k},$$
    whose image, $\Phi(A)\in\S^{12}_{12}+\lambda\Phi(b E_{2j}+c E_{2j'})$ is nonzero in two different columns and, for at least one $\lambda\in{\mathbb k}$, in at least two rows (one of them being the third row). Here and below ${\mathbb k}$ is a prime subfield of $\FF$.
    
    It implies that $\S$ must share exactly two positions in common with $\S^{12}_{12}$, otherwise, for some~$\lambda\in{\mathbb k}$,  $\Phi(A)$ would have three or four nonzero rows (and two nonzero columns) and would not be of per-rank-one.  Moreover, by writing
    $$\Phi(b E_{1j}-c E_{1j'})=\left(\begin{smallmatrix}
                                  \alpha&\alpha\\
                                  \beta&-\beta
                                 \end{smallmatrix}\right)\oplus 0_{n-2}$$
      we see that, if $\S$ and $\S^{12}_{12}$ share  positions $(2,1)$ and $(2,2)$, then $\alpha=0$ while if $\S$ and $\S^{12}_{12}$ share positions $(1,1)$ and $(1,2)$, then $\beta=0$; otherwise  we could again choose $\lambda\in{\mathbb k}$ so that $\Phi(A)$ would be nonzero in three rows and two columns, and would not be of per-rank-one.
      
        Without loss of generality assume that positions $(2,1)$ and $(2,2)$ are shared. Then
        \begin{equation}\label{eq:alpha}
        \alpha=\alpha(bE_{1j}-c E_{1j'})=0.
        \end{equation}
       Hence, if  $\Phi(b E_{2j})$ does not lie in a single line (so  may take $c=0$), then
        $$\Phi(b E_{1j})\in\R_2.$$
        Suppose $\Phi(b E_{2j})$ does  lie in a single line. If we can choose $cE_{2j'}$ so that $\Phi(c E_{2j'})$ also belongs to a single line, while their sum, $\Phi(b E_{2j}+c E_{2j'})$, does not, then $\Phi(b E_{2j}+\mu c E_{2j'})$ will also not belong to a single line for any nonzero $\mu\in{\mathbb k}$. By additivity of $\Phi$, it follows that
        $$0=\alpha(b E_{1j}-\mu c E_{1j'})=\alpha(b E_{1j})-\mu \alpha(c E_{1j'});\qquad \mu\in{\mathbb k}\setminus\{0\}$$
        and since $\mathrm{char}\,\FF\neq2$, we can choose at least two such $\mu\in{\mathbb k}$, giving yet again
        $\Phi(b E_{1j})\in\R_2$.  Finally, if no such  $\Phi(c E_{2j'})$ lies in a single line, then, by repeating the above arguments on $c E_{2j'}$ in place of $bE_{2j}$ (i.e., taking  temporarily $b=0$) reveals that
        $$\alpha(c E_{1j'})=0.$$
        Combined with additivity and~\eqref{eq:alpha}, we see that yet again $\alpha(b E_{1j})=0$, and hence $\Phi(b E_{1j})\in\R_2$. Since $b\in\FF$ and $j\in[n]$ were arbitrary, this shows that $\Phi(\R_1)\subseteq\R_2$, a contradiction.

To finish the proof it only remains to consider the possibility that  $$\Phi(\R_2)\subseteq \R_3.$$
Since $\Phi(\R_1)$ is not contained in a line there exists two weighted matrix units $x E_{1i}$ and $yE_{1j}$ (they  might coincide) in $\R_1$ such that
$$\Phi(x E_{1i}+y E_{1j}) =\left(\begin{smallmatrix}                                                                                                                                                                           \alpha &\alpha\\                                                                                                                                                                        \beta&-\beta                                                                                                                                                                                           \end{smallmatrix}\right)\oplus 0$$
for nonzero $\alpha$ and $\beta$. Then, a   per-rank one matrix
$$ T= \begin{cases}(x E_{1i}+y E_{1j}) +( x E_{2i} - y E_{2j});& xE_{2i}\neq yE_{2j}\\
     (x E_{1i}+y E_{1j})+E_{2j}; & \hbox{otherwise}
    \end{cases}
 $$
satisfies $T \in\R_1+\R_2$ and is mapped into
$$\Phi(T)=\Phi(x E_{1i}+y E_{1j})+\Phi(x E_{2i} - y E_{2j})\in \left(\begin{smallmatrix}                                                                                                                                                                           \alpha &\alpha\\                                                                                                                                                                        \beta&-\beta                                                                                                                                                                                           \end{smallmatrix}\right)\oplus 0+\R_3;\qquad \hbox{if }xE_{2i}\neq yE_{2j}$$
or into
$$\Phi(T)=\Phi(x E_{1i}+y E_{1j})+\Phi(E_{2j})\in \left(\begin{smallmatrix}                                                                                                                                                                           \alpha &\alpha\\                                                                                                                                                                        \beta&-\beta                                                                                                                                                                                           \end{smallmatrix}\right)\oplus 0+\R_3; \qquad \hbox{if }xE_{2i}= yE_{2j}.$$
As such $\Phi(T)$  has nonzero entries in the first two columns as well as in the first three rows, and hence cannot have per-rank one, a contradiction.
\end{proof}
\begin{lemma}\label{lem:TODO} If $\Phi$ maps a line into a nondegenerate square, then it does not preserve per-rank-one in both directions.
    \end{lemma}
    \begin{proof}Assume first there exists a weighted matrix unit whose $\Phi$-image is a nondegenerate square.
       For simplicity (or by   composing $\Phi$ with standard maps in the domain and codomain) we can assume that  
    $$A:=\Phi(E_{11})=\left(\begin{smallmatrix}1&1\\
   1&-1\end{smallmatrix}\right)\oplus 0_{n-2}\in\S^{12}_{12};$$
   so in particular,  $A\in\S^{12}_{12}$.
    Clearly then, the additive subset $\Phi(\R_1)\subseteq\Lambda^{\le 1}$ contains a matrix $A$ of type $T_0$ so by Lemmas~\ref{lem:existence} and~\ref{lemma:max_subsets}, 
    $$\Phi(\R_1)\subseteq \S^{12}_{12}\quad\hbox{or}\quad \Phi(\R_1)\subseteq \dotS^{12}_{12}.$$
   Similarly,
    $$\Phi(\C_1)\subseteq \S^{12}_{12}\quad\hbox{or}\quad \Phi(\C_1)\subseteq \dotS^{12}_{12}.$$
Now, if  $\Phi(\R_1)+\Phi(\C_1)\subseteq\S^{12}_{12}$, then at least one of per-rank-two matrices $E_{12}+ \lambda E_{21}\in\R_1+\C_1$; $\lambda\in\{1,2\}$    would be mapped into $\S^{12}_{12}\setminus\{0\}$.
As this set consists solely of per-rank-one 
matrices such $\Phi$ clearly does not preserve per-rank-one in both directions. 

Assume next $\Phi(E_{12})\in\Phi(\R_1)$ is not contained in $\dotS^{12}_{12}$, and  $\Phi(E_{21})\in\Phi(\C_1)$ is not contained in $\S^{12}_{12}$.  Then,  for every $\mu,\nu\in{\mathbb k}$, the prime subfield of $\FF$, a per-rank-one matrix
$$\mu \nu E_{11}+ \mu  E_{12}+\nu E_{21}-E_{22}$$
    is mapped into a matrix whose upper-left $2$-by-$2$ block equals
$$\mu  \nu  \left(
\begin{array}{cc}
 1 & 1 \\
 1 & -1 \\
\end{array}
\right)+\mu  \left(
\begin{array}{cc}
 \alpha  & \alpha  \\
 \beta  & -\beta  \\
\end{array}
\right)+\nu  \left(
\begin{array}{cc}
 x & y \\
 x & -y \\
\end{array}
\right)-\left(
\begin{array}{cc}
 a   &b  \\
 c & d  \\
\end{array}
\right)$$
for suitable scalars, with $\alpha\neq \beta$ and $x\neq y$.  Its permanent equals
$$-\mu  (\beta  (b-a)+\alpha  (c+d))+\mu  \nu  (a-b-c-d+(\alpha
   -\beta ) (x-y))-\nu  (y (c-a)+x (b+d))+a d+b c,$$
and since there exist at least two nonzero $\mu,\nu\in {\mathbb k}$, after a short computation one sees that the only possibility that per-rank vanishes for every $\mu,\nu\in {\mathbb k}$ is 
$$\Phi( E_{22})=\left(
\begin{array}{cc}
 -x \alpha  & -y \alpha  \\
 -x \beta  & y \beta  \\
\end{array}
\right),$$
where  we used Lemma~\ref{lem:line->squarepart2} by which  $\Phi(\M{n})\subseteq\M{2}\oplus 0_{n-2}$.
Similar arguments for the $\Phi$-image of per-rank-one matrix $\mu\nu E_{11}+\mu E_{13}+\nu E_{21}-E_{23}$ gives
$$\Phi(E_{13})=\begin{pmatrix}
 \gamma  & \gamma  \\
 \delta  & - \delta  \\
\end{pmatrix}
\quad\hbox{and}\quad \Phi( E_{23})=\left(
\begin{array}{cc}
 -x \gamma  & -y\gamma  \\
 -x \delta  & y \delta  \\
\end{array}
\right).$$
Comparing with the $\Phi$-images of per-rank-one matrices $\mu\nu E_{12}+\mu E_{13}+\nu E_{22}-E_{23}$ we actually see that 
$\gamma\beta =\delta\alpha$ so that
$$
\Phi(E_{13})=t \begin{pmatrix}
 \alpha & \alpha  \\
 \beta  & - \beta  \\
\end{pmatrix}$$
We derive a similar conclusion for the image of $E_{31}\in\C_1$. Thus,  it is mapped into
$$\Phi(E_{31})=\epsilon\left(\begin{smallmatrix} x & y\\ x & -y\end{smallmatrix}\right).$$
This shows that the only possibility that $\Phi$ preserves per-rank-one in both directions is
\[
\big(\Phi(E_{ij})\bigl)_{ij}
=
\begin{pmatrix}
\left(\begin{smallmatrix}
1&1\\
1&-1
\end{smallmatrix}\right)
&
\left(\begin{smallmatrix}
\alpha&\alpha\\
\beta&-\beta
\end{smallmatrix}\right)
&
t\left(\begin{smallmatrix}
\alpha&\alpha\\
\beta&-\beta
\end{smallmatrix}\right)
\\[2ex]
\left(\begin{smallmatrix}
x&y\\
x&-y
\end{smallmatrix}\right)
&
-\left(\begin{smallmatrix}
 \alpha  x & \alpha  y \\
 \beta  x & -\beta  y \\
\end{smallmatrix}\right)
&
-t\left(\begin{smallmatrix}
 \alpha   x &\alpha   y \\
 \beta  x &- \beta   y \\
\end{smallmatrix}\right)
\\[2ex]
\epsilon\left(\begin{smallmatrix}
x&y\\
x&-y
\end{smallmatrix}\right)
&
-\epsilon\left(\begin{smallmatrix}
 \alpha   x &\alpha   y \\
 \beta  x &- \beta   y \\
\end{smallmatrix}\right)
&
 -t\epsilon\left(
\begin{smallmatrix}
 \alpha  x   & \alpha   y   \\
 \beta  x   & -\beta  y   \\
\end{smallmatrix}
\right)
\end{pmatrix}.
\]
However, one can show that per-rank-two matrix $E_{21}+E_{32}$ is then mapped into a $2$-by-$2$ matrix of per-rank-one, so $\Phi$ cannot preserve per-rank-one matrices in both directions.

Assume  lastly (as we may) that no  weighted matrix unit is mapped into the nondegenerate square, however, the first row is mapped into a nondegenerate square. Then, after suitable standard maps in domain and codomain we can assume that 
$$\Phi(E_{11})=\left(
\begin{array}{cc}
 1  & 1  \\
 0 & 0 \\
\end{array}
\right)\quad\hbox{and}\quad \Phi(E_{12})=\left(
\begin{array}{cc}
 0  & 0  \\
 1 & -1\\
\end{array}
\right).$$
Hence, $\Phi(\C_1)$ contains a matrix of type $T_2$, so by Lemmas~\ref{lemma:max_subsets} and \ref{lem:46}, $\Phi(\C_1)$ is contained in a square $\S^{12}_{12}$ or in a line $\R_1$. Since we assumed no weighted matrix unit is mapped into a nondegenerate square we must have $\Phi(E_{21})=\left(\begin{smallmatrix}
    \alpha&\beta\\
    0&0\end{smallmatrix}\right)$ or $\Phi(E_{21})=\left(\begin{smallmatrix}
    0&0\\
    \beta&-\beta\end{smallmatrix}\right)$. The last possibility would imply that $\Phi$ would map a per-rank-two matrix $E_{12}+E_{21}$ (if $\beta\neq-1$; or $E_{12}+2E_{21}$ if $\beta=-1$) into $\R_2$ and $\Phi$ would not preserve per-rank-one in both directions. Thus,
    $$\Phi(E_{21})=\left(\begin{smallmatrix}
    \alpha&\beta\\
    0&0\end{smallmatrix}\right)$$
Similar considerations with $E_{12},E_{22}\in\C_2$ give
$\Phi(E_{22})=\left(\begin{smallmatrix}
    0&0\\
    \gamma&\delta\end{smallmatrix}\right)$, and by considering the $\Phi$-image of per-rank-one matrix $\mu \nu E_{11}+\mu E_{12}+\nu E_{21}-E_{22}$, $\mu,\nu\in{\mathbb k}$, we conclude $(\gamma,\delta)=(-\alpha,\beta)$ or else $\gamma=-\delta$ and $\alpha=\beta$. The latter possibility implies that $\Phi(E_{ij})\in\S^{12}_{12}$ for $1\le i,j\le2$, and clearly $\Phi$ does not preserve per-rank-one in both directions. What remains to consider is 
    $$\bigl(\Phi(E_{ij})\bigr)_{ij}
=\left(
\begin{array}{cc}
 \left(
\begin{smallmatrix}
 1 & 1 \\
 0 & 0 \\
\end{smallmatrix}
\right) & \left(
\begin{smallmatrix}
 0 & 0 \\
 1 & -1 \\
\end{smallmatrix}
\right) \\
 \left(
\begin{smallmatrix}
 \alpha & \beta \\
 0 & 0 \\
\end{smallmatrix}
\right) & \left(
\begin{smallmatrix}
0 & 0 \\
 -\alpha & \beta \\
\end{smallmatrix}
\right) \\
\end{array}
\right);\qquad 1\le i,j\le 2.$$
 Proceeding similarly on matrix units in the third column (recall $\Phi(\R_1)\subseteq\S^{12}_{12}$), we see that $\Phi$ will not preserve per-rank-one in both directions except perhaps when  $\Phi(E_{13})=t\left(
\begin{smallmatrix}
0 & 0 \\
 1 & -1 \\
\end{smallmatrix}
\right)$  and $\Phi(E_{23})=\left(
\begin{smallmatrix}
0 & 0 \\
 -\alpha& \beta \\
\end{smallmatrix}
\right)$. Likewise for the third row, giving
$$\bigl(\Phi(E_{ij})\bigr)_{ij}
=\left(
\begin{array}{ccc}
 \left(
\begin{smallmatrix}
 1 & 1 \\
 0 & 0 \\
\end{smallmatrix}
\right) & \left(
\begin{smallmatrix}
 0 & 0 \\
 1 & -1 \\
\end{smallmatrix}
\right) & t\left(
\begin{smallmatrix}
 0 & 0 \\
 1 & -1 \\
\end{smallmatrix}
\right) \\
 \left(
\begin{smallmatrix}
 \alpha  & \beta  \\
 0 & 0 \\
\end{smallmatrix}
\right) & \left(
\begin{smallmatrix}
 0 & 0 \\
 -\alpha  & \beta  \\
\end{smallmatrix}
\right) &t \left(
\begin{smallmatrix}
 0 & 0 \\
 -\alpha   & \beta  \\
\end{smallmatrix}
\right) \\
 \left(
\begin{smallmatrix}
 \gamma  & \delta  \\
 0 & 0 \\
\end{smallmatrix}
\right) & \left(
\begin{smallmatrix}
 0 & 0 \\
 -\gamma  & \delta  \\
\end{smallmatrix}
\right) & t\left(
\begin{smallmatrix}
 0 & 0 \\
 -\gamma  &\delta \\
\end{smallmatrix}
\right) \\
\end{array}
\right).$$
By considering $\Phi(E_{21}+E_{23}+E_{31}-E_{33})$ one sees that this preserves per-rank-one only if $(\gamma,\delta)=s(\alpha,\beta)$ (i.e, they are linearly dependent). However, it then maps per-rank-two matrix $E_{21}+E_{23}+E_{31}$ into $\left(
\begin{smallmatrix}
 \alpha  (s+1) & \beta  (s+1) \\
 -\alpha  t & \beta  t \\
\end{smallmatrix}
\right)\in\Lambda^1$, so again $\Phi$ does not preserve per-rank-one in both directions.
\end{proof}

    We now consider the case when every line is mapped into a line. That is, no row or column is mapped into a nondegenerate square, i.e., a square that is not itself contained in a line. Using standard maps, we may assume without loss of generality that
    \[
    \Phi(\R_1) \subseteq \R_1.
    \]
    Since $\C_1$ intersects $\R_1$ in the domain of $\Phi$, their images must intersect in the codomain as well. Consequently, there are two possibilities:
    \[
    \Phi(\C_1) \subseteq \R_1 \quad \text{or} \quad \Phi(\C_1) \subseteq \C_j \;\; \text{for some }\; j \in [n].
    \]
    We begin with the first case. In the statements of the following lemmas, we implicitly assume that every line is mapped into a line.
 \begin{lemma}\label{lem:R1-C1}
    Let $\Phi \colon \M{n} \to \M{n}$ be an additive map satisfying $\Phi(\Lambda^1) \subseteq \Lambda^1$. Suppose that $\Phi(\R_1) \subseteq \R_1$ and $\Phi(\C_1) \subseteq \R_1$. Then the image of $\Phi$ is contained in a single line.
    \end{lemma}
\begin{proof}
 Suppose, contrary to the claim,  that there exists a weighted matrix unit whose image has a nonzero entry outside the first row. 
    Using standard maps both in the domain and in the codomain we may assume that
    \begin{equation}\label{eq:Phi(xE22)}
    b_{22}:= \bigl(\Phi( E_{22})\bigr)_{21} \neq 0.
    \end{equation}
    Since $\Phi(\FF E_{21})\subseteq\Phi(\R_2)\cap\Phi(\C_1)\subseteq\Phi(\R_2)\cap\R_1\neq\{0\}$, and by~\eqref{eq:Phi(xE22)},
    the only maximal additive subsets of $\Lambda^{\le1}$ that can contain the image of $\R_2$  are the first column  or different full squares. The latter possibility would imply that a line is mapped into a full square, a case that has already been treated earlier.   Similar  arguments with the same conclusions hold for the image of $\C_2$. Hence,
    \begin{equation}\label{eq:Phi(R2)inC1}
    \Phi(\R_2) \subseteq \C_1 \quad \text{and} \quad
    \Phi(\C_2) \subseteq \C_1.
    \end{equation}
Take now a row $\R_3$, which intersects $\C_1$ as well as $\C_2$ in per-rank-one matrices. Then, its image must intersect $\Phi(\C_1)\subseteq\R_1$ as well as $\Phi(\C_2)\subseteq\C_1$ in per-rank-one matrices. Since $\Phi(\R_3)$ is contained in a line, the only two possibilities are that
$$\Phi(\R_3)\subseteq\C_1\quad\hbox{ or }\quad\Phi(\R_3)\subseteq\R_1.$$
Likewise we see that
$$\Phi(\R_i),\Phi(\C_j)\subseteq\C_1\cup\R_1;\qquad i,j\ge 3.$$
Hence, by additivity,
$$\Phi(\M{n})\subseteq\R_1+\C_1.$$ 
Now, if $\Phi(\M{n})\subseteq\C_1$ there is nothing to do. Otherwise, there must exist a weighted matrix unit $yE_{ij}\notin\C_2+\R_2$, with $\Phi(yE_{ij})\in\R_1\setminus \C_1$. Again, we can  compose $\Phi$ with  standard maps in the domain, which fix $E_{22}$, to achieve that $y E_{ij}=E_{11}$.  Observe that the changed $\Phi$ satisfies $\Phi(E_{22})\in\C_1\setminus\R_1$ (by~\eqref{eq:Phi(xE22)}) and $\Phi(E_{11})\in\R_1\setminus\C_1$, and so
\begin{equation}\label{eq:aux}    \Phi(\C_2)\subseteq\C_1\setminus\R_1\hbox{ and } \Phi(\R_1)\subseteq\R_1\setminus\C_1.
\end{equation}

With no loss of generality assume $b_{11}:=\bigl(\Phi(E_{11})\bigr)_{12}\neq0$; otherwise we permute column $2$ and column $i\ge3$ in codomain to achieve this. Consider now  per-rank-one matrices $\mu\nu E_{11}+\mu E_{12}+\nu E_{21}-E_{22}$ for $\mu,\nu\in{\mathbb k}$, the prime subfield of $\FF$. They are mapped into
\begin{equation}\label{eq:aux2}
    \mu\nu \Phi(E_{11})+\mu\Phi(E_{12})+\nu\Phi(E_{21})-\Phi(E_{22}).
\end{equation}
The first summand is contained in $\FF\Phi(E_{11})\subseteq\R_1\setminus\C_1$. The second and third are contained in 
$$\Phi(\R_1)\cap\Phi(\C_2)\subseteq\R_1\cap\C_1=\FF E_{11}\hbox { and }\Phi(\R_2)\cap\Phi(\C_1)\subseteq\C_1\cap\R_1=\FF E_{11}$$ (for the first one use~\eqref{eq:aux}, for the second one use~$\Phi(E_{22})\in\Phi(\R_2)$ and $\Phi(E_{11})\in\Phi(\R_1)$). Hence, the compression to $2$-by-$2$ block of~\eqref{eq:aux2} equals
$$\mu  \nu  \left(
\begin{array}{cc}
 a_{11} & b_{11} \\
 0 & 0 \\
\end{array}
\right)+\mu  \left(
\begin{array}{cc}
 a_{12} & 0 \\
 0 & 0 \\
\end{array}
\right)+\nu  \left(
\begin{array}{cc}
 a_{21} & 0 \\
 0 & 0 \\
\end{array}
\right)-\left(
\begin{array}{cc}
 a_{22} & 0 \\
 b_{22} & 0 \\
\end{array}
\right)$$
and its permanent, $-\mu  \nu b_{11} b_{22} $ is clearly nonzero, a contradiction.

\end{proof}
    \begin{lemma}\label{lem:sigma-tau}
    Assume $\Phi$ maps rows to rows and columns to columns. There exist  transformations $\sigma, \tau\colon[n]\to[n]$ and additive functions $f_{11}, f_{12}, \ldots, f_{nn}\colon\F\to\F$ such that
   \begin{equation}
         \Phi((a_{ij}))  = P_\sigma^T\bigl(f_{ij}(a_{ij})\bigr)_{ij}P_\tau \quad \forall A = (a_{ij}) \in \M{n}.
     \label{eq:l4.11}
 \end{equation}
Moreover, if $\Im\Phi$ is not contained in a single row or a single column, then there exist invertible diagonal $D_1,D_2$ such that
$$\Phi((a_{ij}))  = P_\sigma^TD_1\bigl(g_{ij}(a_{ij})\bigr)_{ij} D_2P_\tau$$ for some additive $g_{ij}\colon\FF\to\FF$ with $g_{ij}(1)=1$.
\end{lemma}
    \begin{proof}
    By the assumptions, there exist transformations $\sigma,\tau\colon[n]\to[n]$  with
    $$\Phi(\R_i)\subseteq\R_{\sigma(i)}\quad\hbox{ and }\quad \Phi(\C_j)\subseteq \C_{\tau(j)}.$$
Since $\FF E_{ij}=\R_i\cap \C_j$ we then have
$\Phi(\FF E_{ij})\subseteq\Phi(\R_i)\cap\Phi(\C_j)\subseteq\R_{\sigma(i)}\cap \C_{\tau(j)}=\FF E_{\sigma(i)\tau(j)}$.
In particular,
$$\Phi(x E_{ij})=f_{ij}(x) E_{\sigma(i)\tau(j)}$$ for some additive functions $f_{ij}\colon\F\to
\F$.
By additivity this gives that $A=(a_{ij})\in \M{n}$ is mapped into
$$\Phi(A)=\Phi(\sum_{ij} a_{ij} E_{ij}) =\sum_{ij} f_{ij}(a_{ij}) E_{\sigma(i)\tau(j)}=P_{\sigma}^T(f_{ij}(a_{ij}))P_\tau. $$

To prove the last claim, we suppose $\operatorname{Im}\Phi$ is not contained in a single line. Since $\operatorname{Im}\Phi$ is then not contained in a single row (respectively, column), there exist $i_1\neq i_2$ with $\sigma(i_1)\neq\sigma(i_2)$ (respectively, $j_1\neq j_2$ with $\tau(j_1)\neq\tau(j_2)$); composing $\Phi$ with suitable row- and column-permutations of the domain and codomain, we may thus assume for simplicity and with no loss of generality that $\Phi(\R_i)\subseteq \R_i$ and $\Phi(\C_j)\subseteq \C_j$ (so that $\sigma(i)=i$ and $\tau(j)=j$) for $i,j=1,2$.
Then,
 \[
    \det\begin{pmatrix} f_{11}(1) & f_{12}(1) \\ f_{21}(1) & f_{22}(1) \end{pmatrix} = \per \begin{pmatrix} f_{11}(1) & f_{12}(1) \\ f_{21}(-1) & f_{22}(1) \end{pmatrix} = 0
    \]
since $\left(\begin{smallmatrix} f_{11}(1) & f_{12}(1) \\ f_{21}(-1) & f_{22}(1) \end{smallmatrix} \right)\oplus 0_{n-2} = \Phi(\left(\begin{smallmatrix} 1 & 1 \\ -1 & 1 \end{smallmatrix}\right)\oplus 0_{n-2})$ and $\prk(\left(\begin{smallmatrix} 1 & 1 \\ -1 & 1 \end{smallmatrix}\right)\oplus 0_{n-2})=1$.
It follows that the two vectors $(f_{11}(1),f_{12}(1))$ and $(f_{21}(1),f_{22}(1))$ are linearly dependent.

Now, take  $E_{13}$. By the assumptions, $\Phi(\C_3)\subseteq\C_{j'}$ and  at least one among indices $1,2$ differs from $j'$. With no loss of generality,  $j'\neq 1$ and by temporarily   permuting columns of $\Phi(\cdot)$ with a suitable permutation that fixes $1$ we achieve $j'=2$. Arguing as above,
$$\det\left(\begin{smallmatrix}
f_{11}(1) & f_{13}(1) \\ f_{21}(1) & f_{23}(1) \end{smallmatrix} \right)=0$$
and since $f_{ij}(1) \neq 0$ (by $f_{ij}(1)E_{\sigma(i)\tau(j)}=\Phi(E_{ij})\neq0$),  the truncated rows $(f_{11}(1),f_{12}(1),f_{13}(1))$ and $(f_{21}(1),f_{22}(1),f_{23}(1))$ are linearly dependent. Proceeding recursively, the first two rows of $(f_{ij}(1))$ are linearly dependent.

Consider matrix units from the  third row $E_{31},E_{32},\dots\in\R_3$. They are mapped into $\R_{i'}$ and at least one among indices $1,2$ differs from $i'$; again we can assume $i'\neq 1$. By applying the above arguments   we see that the first and the third row of $(f_{ij}(1))$ are linearly dependent. Continuing, we get that all its rows are linearly dependent so $\mathrm{rk}(f_{ij}(1))=1$. We can write it as
$$(f_{ij}(1))_{ij}=xy^T;$$
for some $x=(x_1,\dots,x_n)^T$ and $y=(y_1,\dots,y_n)^T$, and since $f_{ij}(1)\neq0$,
then $D_1:=\diag(x)$ and $D_2:=\diag(y)$ and $g_{ij}=\frac{1}{x_iy_j}f_{ij}$ are as required.
    \end{proof}
    \begin{lemma}\label{lem:final-onedireciton}
    Assume $\Im \Phi$ is not contained in a single line.  Then, the functions $g_{11} = g_{12} = \cdots = g_{nn} \eqqcolon \varphi$ coincide and $\varphi$ is a field homomorphism of~$\F$.
    \end{lemma}
    \begin{proof}
    Again we can assume that $\Phi(\R_i)\subseteq\R_i$ and $\Phi(\C_j)\subseteq\C_j$ (or equivalently, $\sigma(i)=i$ and $\tau(j)=j$) for $i,j=1,2$. Choose $x\in\FF$ and consider a matrix $\left(\begin{smallmatrix} x & x \\ 1 & -1 \end{smallmatrix}\right)\oplus 0_{n-2}$ of per-rank one. Its $\Phi$-image, which by Lemma~\ref{lem:sigma-tau} equals
    $$P_\sigma^T D_1\left(\left(\begin{smallmatrix} g_{11}(x) & g_{12}(x) \\ 1 & -1 \end{smallmatrix}\right)\oplus 0_{n-2}\right)D_2P_\tau=D_1\left(\left(\begin{smallmatrix} g_{11}(x) & g_{12}(x) \\ 1 & -1 \end{smallmatrix}\right)\oplus 0_{n-2}\right)D_2$$  is again of per-rank one, and so
     $$0 = \per \left(\begin{matrix} g_{11}(x) & g_{12}(x) \\ 1 & -1 \end{matrix}\right) = g_{12}(x) - g_{11}(x).$$
     Next, the column which contains  $\Phi(\C_3)$ is disjoint  from $\C_1$ or from $\C_2$ and by the same procedure we get $g_{11}(x)=g_{12}(x)=g_{13}(x)$. Repeated application of this procedure yields the conclusion that $g_{ij}(x) = g_{kl}(x)=:\varphi(x)$ for all $x\in\FF$ and $i, j, k, l\in[n]$.
     
    Now consider $\left(\begin{smallmatrix} -1 & a \\ b & a b \end{smallmatrix}\right)\oplus0_{n-2}\in\Lambda^1$. Its $\Phi$-image, $$D_1\left(\left(\begin{smallmatrix} -1 & \varphi(a) \\ \varphi(b) & \varphi(a b) \end{smallmatrix}\right)\oplus0_{n-2}\right)D_2$$ has to have a zero permanent. 
    Thus, $0 = -\varphi(a b) + \varphi(a) \varphi(b)$, 
    and an additive $\varphi$ is a field endomorphism.
    \end{proof}
    \begin{proof}[Proof of  Theorem \ref{th:prk1_main}]
    If $\Phi$ maps a line into a  full square, then, by composing it with standard maps in the domain and codomain, we can achieve that the first row, $\R_1$ goes into a full square $\S^{12}_{12}$. In this case, by Lemma~\ref{lem:line->squarepart2}, $\Phi(\M{n})\subseteq(\R_1+\R_2)\cap(\C_1+\C_2)$, i.e., $\Phi$ takes the last of the degenerate possibilities. 
    
    Assume $\Phi$ is not degenerate. Then, as  just shown, it maps lines to lines. Hence, by composing it with  standard maps we can assume that
    $$\Phi(\R_1)\subseteq\R_1,$$
    and since $\Phi$ is nondegenerate, Lemma~\ref{lem:R1-C1} implies that $\Phi(\C_1)$ cannot be contained in a row (and cannot be mapped into a full square), so it is mapped into a column. We can temporarily replace $\Phi$ by its composition  with a suitable permutation of columns in the domain to see, by exactly the same arguments, that every column is mapped into a column. We may further temporarily replace $\Phi$ by its composition with the transposition map in the domain, to see that also rows are mapped into rows. In particular, there are transformations $\sigma,\tau\colon[n]\to[n]$ such that
    $$\Phi(\R_i)\subseteq \R_{\sigma(i)}\quad\hbox{and}\quad \Phi(\C_j)\subseteq\C_{\tau(j)}.$$
   By Lemma~\ref{lem:sigma-tau} this implies that $\Phi$ has the form \eqref{eq:l4.11} for some invertible diagonal matrices $D_1,D_2$. Then by Lemma~\ref{lem:final-onedireciton} it holds that $g_{11}=g_{12}=\ldots=g_{nn}=\varphi$ for some   field endomorphism $\varphi\colon\FF\to\FF$ and
    we have that
    $$\Phi\colon X\mapsto P_\sigma^{T} D_1X^\varphi D_2 P_\tau.\qedhere$$
   \end{proof}
    \begin{proof}[Proof of Theorem~\ref{th:prk1_maina}]  For surjective maps this is shown in Remark~\ref{rem:phisurjective}. However, we prefer to give a self-contained short proof as follows:
    
    By Lemma~\ref{lem:existence}, the maximal additive subsets $\R_i\subseteq \Lambda^{\le1}$ are mapped into some maximal additive subsets of $\Lambda^{\le1}$. By Lemma~\ref{lemma:max_subsets}, these are vector subspaces of dimensions two (for a full square) or $n$ (for a line). Now,
    \begin{equation}
    \label{dimension}
    n^2 = \dim \M{n} = \dim \Phi(\M{n}) \leq \sum_{i = 1}^{n} \dim \mathrm{span}_{\F}  \Phi(\R_i)  \leq n \cdot n = n^2
    \end{equation}
    Consequently all inequalities in \eqref{dimension} are equalities, so we have $\dim \mathrm{span}_{\F}  \Phi(\R_i)  = n$ for all $i$ and $\mathrm{span}_{\F} \Phi(\R_i)  \cap \mathrm{span}_{\F}  \Phi(\R_j)  = 0$ for all $i \neq j$. Since $n > 2$ each row must be mapped into a row or a column.
    If $\Phi(\R_i)\subseteq\R_{i'}$ and $\Phi(\R_j)\subseteq \C_{j'}$, then $\mathrm{span}_{\F}\Phi(\R_i)=\R_{i'}$ and $\mathrm{span}_{\F}\Phi(\R_j)=\C_{j'}$ would intersect in a cell $\FF E_{i'j'}$, a contradiction. Hence,  if one row was mapped into a row, then all rows were mapped into different rows, and likewise if one row was mapped into a column. By surjectivity, rows are actually mapped onto rows or onto columns. Same holds for images of columns. Then, cells $\F E_{ij}=\R_i\cap \C_j$ are mapped into cells. By Lemma~\ref{lem:sigma-tau} this implies that $\Phi$ has the form \eqref{eq:l4.11}. Then by Lemma~\ref{lem:final-onedireciton} it holds that $g_{11}=g_{12}=\ldots=g_{nn}=\varphi$ and the result follows. 
    
    If $\Phi$ is perhaps nonsurjective but preserves per-rank-one in both directions, then we  rely on Theorem~\ref{th:prk1_main}.  Clearly, the first two degenerate maps cannot preserve per-rank one in both directions, since they map every matrix to a matrix of per-rank at most one, and at least one among $E_{11}+\alpha E_{22}$, $\alpha=\pm1$ is not annihilated. Also,  $\Phi$ cannot map a line into nondegenerate square, by Lemma~\ref{lem:TODO}.   So, it suffices to show that
    $$X\mapsto P_{\sigma}^TD_1X^\varphi D_2 P_\tau$$ preserves per-rank one in both directions if and only if transformations $\sigma$ and $\tau$ are injective. Assume otherwise that, say $\sigma$ is not injective. For simplicity, let $\sigma(1)=\sigma(2)=1$. Then, $\R_1$ and $\R_2$ are mapped into $\R_1$, and, by additivity,  both matrices $A=\left(\begin{smallmatrix}1 &1\\ \alpha& \alpha\end{smallmatrix}\right)\oplus 0_{n-2}\in\R_1+\R_2$ (for $\alpha=\pm1$), of per-rank two are mapped into  matrices in $\R_1$, at least one of them nonzero, a contradiction.  Likewise we see that $\tau$ must be injective.
    \end{proof}

\noindent Alexander Guterman\\  Department of Mathematics, Bar-Ilan University, Ramat-Gan 5290002, Israel, alexander.guterman@biu.ac.il \\  Bojan Kuzma\\ University of Primorska, Glagoljaška 8, SI-6000 Koper, Slovenia;
IMFM, Jadranska 19, SI-1000 Ljubljana, Slovenia, bojan.kuzma@upr.si\\
Leonid Ovchinnikov\\ University of Primorska, Glagoljaška 8, SI-6000 Koper, Slovenia,\\ leonid.ovchinnikov3.14@gmail.com

\end{document}